\documentclass[12pt,reqno]{amsart}

\usepackage{latexsym, amsmath, amssymb, amscd}

\usepackage{mathrsfs}

\newtheorem{theorem}{Theorem}[section]

\newtheorem{proposition}[theorem]{Proposition}

\theoremstyle{theorem}
\newtheorem{maintheo}[theorem]{Main Theorem}

\theoremstyle{definition}
\newtheorem{definition}[theorem]{Definition}

\theoremstyle{remark}
\newtheorem{remark}[theorem]{Remark}

\numberwithin{equation}{section}

\newcommand{\ts}{\hspace{.11111em}}
\newcommand{\tts}{\hspace{.05555em}}

\newcommand{\R}{\mathbb{R}}
\newcommand{\T}{\mathbb{T}}
\newcommand{\Z}{\mathbb{Z}}

\DeclareMathOperator{\G}{\operatorname{\mathcal{G}}}  
\DeclareMathOperator{\length}{\operatorname{\mathsf{Length}}\tts}
\DeclareMathOperator{\area}{\operatorname{\mathsf{Area}}\tts}
\DeclareMathOperator{\vol}{\operatorname{\mathsf{Vol}}\tts}

\newcommand{\RP}{\mathbb{R}\mathbb{P}}
\newcommand{\C}{\mathbb{C}}

\DeclareMathOperator{\sys}{\operatorname{\mathsf{Sys}}\tts}
\DeclareMathOperator{\stsys}{\operatorname{\mathsf{Stsys}}\tts}
\DeclareMathOperator{\SR}{\operatorname{\mathsf{SR}}\tts}

\DeclareMathOperator{\Mod}{\operatorname{\mathsf{Mod}}\tts}

\begin{document}

\title{$\Z_{2}$-coefficient homology $(1, 2)$-systolic freedom of $\RP^{3} \# \RP^{3}$}

\author[L.~Chen]{Lizhi Chen}

\thanks{}



\address{Department of Mathematics, Oklahoma State University \newline 
\hspace*{0.175in} Stillwater, OK 74078}

\curraddr{Chern Institute of Mathematics, Nankai University \newline 
\hspace*{0.175in} Tianjin 300071, P.R. China
}

\email{lizhi@ostatemail.okstate.edu}



\date{October 11, 2014.}

\begin{abstract}
We prove the $3$-manifold $\RP^3 \# \RP^3$ is of $\Z_{2}$-coefficient homology $(1, 2)$-systolic freedom. Given a Riemannian metric on $\RP^{3}\# \RP^{3}$, we define $\Z_{2}$-coefficient homology $1$-systole as the infimum of lengths of all nonseparating geodesic loops representing nontrivial classes in $H_{1}(\RP^3\#\RP^3; \Z_{2})$. The $\Z_{2}$-coefficient homology $2$-systole is defined to be the infimum of areas of all nonseparating surfaces representing nontrivial classes in $H_{2}(\RP^{3}\#\RP^{3}; \Z_2)$. In the paper we show that there exists a sequence of Riemannian metrics on $\RP^{3} \# \RP^{3}$ such that the volume of $\RP^3 \# \RP^3$ cannot be bounded below in terms of the product of $\Z_{2}$-coefficient homology $1$-systole and $\Z_{2}$-coefficient homology $2$-systole.
\end{abstract}

\maketitle 

\tableofcontents

\section{Introduction}

Let $M$ be a compact $3$-manifold with the Riemannian metric $\mathcal{G}$, denoted $(M, \G)$. The length of a loop $\gamma$ in $(M, \G)$ is denoted $\length_{\G}(\gamma)$, and the area of a surface $\Sigma$ in $(M, \G)$ is denoted $\area_{\G}(\Sigma)$. Let $\Z$ be the ring of all integers and $\Z_{2}$ be the ring $ \Z / 2\Z $. We use the convention that a $3$-manifold is closed if it is compact and without boundary.
\begin{definition}
Let $(M, \G)$ be a closed compact Riemannian $3$-manifold. 
\begin{enumerate} 
 \item The $\Z_{2}$-coefficient homology $1$-systole of $(M, \G)$, denoted 
 \[ \sys H_{1}(M, \G; \Z_{2}), \]
 is defined as
  \begin{equation*}
   \inf_{\ell} \tts \length_{\G}(\ell),
  \end{equation*}
  where the infimum is taken over all nonseparating loops $\ell$ which represent nontrivial classes in $H_{1}(M; \Z_{2})$.
  \item The $\Z_{2}$-coefficient homology $2$-systole of $(M, \G)$, denoted 
   \[ \sys H_{2}(M, \G; \Z_{2}) , \]
   is defined as
   \begin{equation*}
    \inf_{\Sigma} \tts \area_{\G}(\Sigma),
   \end{equation*}
   where the infimum is taken over all nonseparating surfaces $\Sigma$ which represent nontrivial classes in $H_{2}(M; \Z_{2})$.
 \end{enumerate}
\end{definition}
In this paper, we investigate the question of whether the Riemannian volume of a $3$-manifold $M$ can be bounded below in terms of its $\Z_{2}$-coefficient homology systoles.
\begin{definition}
 A closed compact $3$-manifold is of $\Z_{2}$-coefficient homology $(1, 2)$-systolic freedom if
 \begin{equation*}
  \inf_{\G} \tts \frac{\vol_{\G}(M)}{\sys H_{1}(M, \G; \Z_{2}) \cdot \sys H_{2}(M, \G; \Z_{2})} = 0,
 \end{equation*}
 where the infimum is taken over all Riemannian metrics $\G$ on $M$.
 \label{Z2_freedom}
\end{definition}
Let $\RP^3 \# \RP^3$ be the connected sum of two real projective $3$-spaces. We prove the following result in this paper.
\begin{maintheo}
 The $3$-manifold $\RP^{3} \# \RP^{3}$ is of $\Z_{2}$-coefficient homology $(1, 2)$-systolic freedom.
 \label{Chen14}
\end{maintheo}

Previously in 1999, Freedman \cite{Fr99} proved the following result of $\Z_2$-coefficient homology $(1, 2)$-systolic freedom.
\begin{theorem}[Freedman \cite{Fr99}]
 The $3$-manifold $S^{2} \times S^{1}$ is of $\Z_{2}$-coefficient homology $(1, 2)$-systolic freedom.
\end{theorem}
Freedman's theorem provides the first answer to the following question. Let $M$ be a closed compact $3$-manifold with nontrivial $H_1(M; \Z_2)$ and $H_2(M; \Z_2)$. It is not known if there exists a positive constant $C$ so that
\begin{equation*}
 \inf_{\G} \tts \frac{\vol_{\G}(M)}{\sys H_{1}(M, \G; \Z_{2}) \cdot \sys H_{2}(M, \G; \Z_{2})} \geqslant C ,
\end{equation*}
where the infimum is taken over all Riemannian metrics $\G$ on $M$. Gromov conjectured the existence of such a positive constant $C$, see \cite{KaSu98} or \cite{Fr99}. However, Freedman's theorem gives a counterexample to Gromov's conjecture. Our theorem 1.3 is the second result for $\Z_2$-coefficient homology $(1, 2)$-systolic freedom of $3$-manifolds. 

Similarly, the $\Z$-coefficient homology $1$-systole of a Riemannian $3$-manifold $(M, \G)$, denoted $\sys H_{1}(M, \G; \Z)$, is defined to be the infimum of lengths of all nonseparating loops which represent nontrivial classes in $H_{1}(M; \Z)$. The $\Z$-coefficient homology $2$-systole of $(M, \G)$, denoted
\[ \sys H_{2}(M, \G; \Z) , \]
is defined to be the infimum of areas of all nonseparating surfaces which represent nontrivial classes in $H_{2}(M; \Z)$. The $\Z$-coefficient homology $(1, 2)$-systolic freedom of $(M, \G)$ is defined in a way similar to Definition \ref{Z2_freedom}, see Babenko and Katz \cite[1.6]{BaKa98}.
\begin{theorem}[Bergery and Katz \cite{BeKa94}]
 The $3$-manifold $S^{2} \times S^{1}$ is of $\Z$-coefficient homology $(1, 2)$-systolic freedom.
\end{theorem}
Moreover, according to Babenko, Bergery and Katz \cite{BaKa98}, the following theorem is obtained.
\begin{theorem}[Babenko and Katz \cite{BaKa98}]
 Let $M$ be a compact orientable $3$-manifold. Then $M$ is of $\Z$-coefficient homology $(1, 2)$-systolic freedom.
\end{theorem}
Katz and Suciu's result in \cite{KaSu01} leads to the following theorem.
\begin{theorem}[Katz and Suciu \cite{KaSu01}]
 For an orientable $3$-manifold $M$ with $H_{2}(M; \Z)$ torsion free,
 \begin{equation*}
  \inf_{\G} \tts \frac{\vol_{\G}(M)}{\sys H_{2}(M, \G; \Z)^{3/2} } = 0,
 \end{equation*}
 where the infimum is taken over all Riemannian metrics $\G$ on $M$.
\end{theorem}
The $\Z$-coefficient homology systolic freedom widely exists, see Gromov \cite[4.45 and Appendix D]{Gr07}, B{\'e}rard-Bergery and Katz \cite{BeKa94}, Katz \cite{Ka95, Ka02}, Pittet \cite{Pi97}, Babenko and Katz \cite{BaKa98}, Babenko and Katz and Suciu \cite{BaKaSu98}, Katz and Suciu \cite{KaSu98, KaSu01}, Babenko \cite{Ba02, Ba02b}.  

 The paper is organized as follows. In section 2, we have a brief introduction to systolic geometry. In section 3, we introduce some preliminary knowledge of $3$-manifolds with surface bundle strucutre. The definition of Dehn surgery is also introduced in this section. In section 4, we review Freedman's technique in the proof of $\Z_{2}$-coefficient homology $(1, 2)$-systolic freedom of $S^{2} \times S^{1}$. In section 5, we investigate properties of $\Z_{2}$-coefficient homology systoles for $3$-manifolds with semibundle structure. Section 6 contains the proof of Main Theorem \ref{Chen14}. 

\section{A brief review of systolic geometry}

Let $M$ be a manifold with the Riemannian metric $\G$, denoted $(M, \G)$. 
\begin{definition}
 The homotopy $1$-systole of $(M, \G)$, denoted $\sys \pi_{1}(M, \G)$, is defined to be the infimum of lengths of all noncontractible loops in $M$. 
\end{definition}

The research of systolic geometry is initiated by Loewner and Pu. Loewner proved the first systolic inequality for homotopy $1$-systole on torus.
\begin{theorem}[Loewner, see Pu \cite{Pu52} or Katz \cite{Ka07}]
 For every Riemannian metric $\G$ defined on a torus $\T^{2}$,
 \begin{equation}
  \sys \pi_{1}(\T^{2}, \G)^{2} \leqslant \frac{2}{\sqrt{3}} \vol_{\G}(\T^{2}) ,   
  \label{Loewner}
 \end{equation}
 where equality holds for a metric $\G$ realizing a flat hexagonal torus. 
\end{theorem}
After Loewner, Pu \cite{Pu52} proved another inequality for homotopy $1$-systole on real projective plane $\RP^{2}$. 
\begin{theorem}[Pu \cite{Pu52}]
 For every Riemannian metric $\G$ defined on a real projective plane $\RP^{2}$,
 \begin{equation}
  \sys \pi_{1}(\RP^{2}, \G)^{2} \leqslant \frac{\pi}{2} \ts \vol_{\G}(\RP^{2}),
  \label{Pu}
 \end{equation}
 where equality holds for metrics $\G$ with constant curvature.
\end{theorem}
Define the optimal systolic ratio $\SR(M)$ of an $n$-dimensional manifold $M$ to be
\begin{equation*}
 \inf_{\G} \tts \frac{\vol_{\G}(M)}{\sys \pi_{1}(M, \G)^{n}},
\end{equation*}
where the infimum is taken over all Riemannian metrics $\G$ on $M$. Loewner inequality (\ref{Loewner}) implies that $\SR(\T^{2}) = \frac{\sqrt{3}}{2}$, and $\SR(\RP^{2}) = \frac{2}{\pi}$ by Pu inequality (\ref{Pu}). Moreover, Bavard \cite{Ba86} proved that the optimal systolic ratio of Klein bottle $\RP^{2} \# \RP^{2}$ is equal to $\frac{2\sqrt{2}}{\pi}$, which is realized by a singular metric. Based on Pu inequality and Gromov's inequality \cite[Corollary 5.2.B.]{Gr83}, Croke and Katz \cite{CroKa03} summarized that
\begin{equation*}
 \SR(\Sigma) \geqslant \frac{2}{\pi}
\end{equation*} 
for any closed surface $\Sigma$ other than the $2$-sphere $S^{2}$. For a closed surfaces $\Sigma$ other than $S^{2}, \T^{2}, \RP^{2}, \RP^{2} \# \RP^{2}$, the optimal systolic ratio is unknown.

A topological space $K$ is aspherical if all higher homotopy groups $\pi_{i}(K)$ vanish, with $i \geqslant 2$. An $n$-dimensional manifold $M$ is essential if there exists a map $f: M \to K$ to an aspherical topological space $K$ such that $f_{*}([M]) \neq 0$ with the fundamental class $[M] \in H_{n}(M; \Z)$ when $M$ is orientable, and with the fundamental class $[M] \in H_{n}(M; \Z_{2})$ when $M$ is nonorientable. Gromov \cite{Gr83} showed the existence of systolic inequality for homotopy $1$-systole of essential manifolds.
\begin{theorem}[Gromov \cite{Gr83}]
 Let $M$ be an essential $n$-dimensional manifold. For any Riemannian metric $\G$ defined on $M$,
 \begin{equation*}
  \sys \pi_{1}(M, \G)^{n} \leqslant C(n) \tts \vol_{\G}(M),
 \end{equation*}
 where $C(n) = \left( \tts 6 (n+1) \cdot n^{n} \cdot \sqrt{(n+1)!} \ts \right)^{n}$.
\end{theorem}   

The $\Z$-coefficient (or $\Z_{2}$-coefficient) homology $k$-systole of a Riemannian $n$-dimensional manifold $(M, \G)$ is defined to be the infimum of volumes of all cycles representing nontrivial $k$-classes of $H_{k}(M; \Z)$ (or $H_{k}(M; \Z_{2})$), denoted $\sys H_{k}(M, \G; \Z)$ (or $\sys H_{k}(M, \G; \Z_{2})$). For $\Z$-coefficient (or $\Z_{2}$-coefficient) homology $k$-systoles, there are phenomena of systolic freedom, see Section 1 of this paper and references therein. The norm of a homology class $\alpha \in H_{k}(M; \Z)$, denoted $\| \alpha \|$, is defined to be
\begin{equation*}
 \inf_{c} \vol_{\G}(c),
\end{equation*} 
where the infimum is taken over all cycles $c$ representing $\alpha$. The stable norm of $\alpha \in H_{k}(M; \Z)$, denoted $\| \alpha \|_{s}$, is defined to be
\begin{equation*}
 \lim_{i \to \infty} \frac{\| i \alpha \|}{i}.
\end{equation*}
\begin{definition}
 The stable $k$-systole of $(M, \G)$, denoted $\stsys_{k}(M, \G)$, is defined to be
 \begin{equation*}
  \inf_{\alpha} \|\alpha\|_{s},
 \end{equation*}
 where the infimum is taken over all nontrivial homology classes $\alpha$ in $H_{k}(M; \Z)$.
\end{definition}
\begin{theorem}[Gromov \cite{Gr83}]
 Let $M$ be a connected and closed orientable $n$-dimensional manifold. Let $(k_{1}, k_{2}, \cdots , k_{p})$ be a partition of $n$, i.e. an unordered sequence of positive integers such that $n = \sum_{i = 1}^{p} k_{i}$. If there are cohomology classes $\beta_{i} \in H^{k_{i}}(M; \R)$ with nonzero cup product $\beta_{1} \cup \beta_{2} \cup \cdots \cup \beta_{p} \in H^{n}(M; \R)$, then for every Riemannian metric $\G$ on $M$,
\begin{equation*}
 \prod_{i = 1}^{p} \stsys_{i}(M, \G) \leqslant C \tts \vol_{\G}(M),
\end{equation*} 
 where $C$ is a positive constant which only depends on the dimension $n$, the partition $(k_{1}, k_{2}, \cdots , k_{p})$ and the Betti numbers $b_{k_{i}}(M)$.
\end{theorem} 
More results of stable systoles can be found in Brunnbauer \cite{Bru08}, Bangert and Katz \cite{BanKa03}, Hebda \cite{Heb86}. 

Expository of systolic geometry and topology include Berger \cite[7.2]{Ber03}, Croke and Katz \cite{CroKa03}, Gromov \cite{Gr92}, Katz \cite{Ka07}, Guth \cite{Gu10}, Chen and Li \cite{CheLi14}. For recent progress, we can see Katz et al \cite{Katzetal11}, Katz and Sabourau \cite{KaSa12, KaSa12b}, Belolipetsky \cite{Bel13}, Belolipetsky and Thomson \cite{BeTho11}, Nakamura \cite{Na13}, Lakeland and Leininger \cite{LaLei14}, {\'A}lvarez Paiva and Balacheff \cite{PaiBa14}.

\section{$3$-manifolds with surface bundle structure}

In this section, we introduce some preliminary knowledge of 3-manifolds with surface bundle structure, i.e., 3-manifolds which fiber over the cicrcle. The class of 3-manifolds with surface bundle structure plays an important role in the theory of 3-manifolds, see \cite{Ag13, Thu98}. Thurston's theory establishes a geometric classification for $3$-manifolds with surface bundle structure. The $3$-manifold which fibers over circle has a mapping torus structure. The classification theorem of Thurston is based on the classification of the monodromy of mapping torus. In order to explain Thurston's theorem, we introduce the mapping class group of surfaces. The Dehn-Lickorish theorem implies that the mapping class group of surfaces is generated by Dehn twists, which is a fundamental theorem in mapping class group theory. Finally we introduce Dehn surgery of $3$-manifolds, which is used in the construction of metrics exhibiting systolic freedom, see Section 4 and Section 6.

Let $S$ be a connected and closed orientable surface. In this paper, we use the convention that a surface is closed if it is compact and without boundary. We use $\text{Diff}^{+}(S)$ to denote the group of orientation preserving diffeomorphisms. And $\text{Diff}_{0}(S)$ stands for the subgroup of $\text{Diff}^{+}(S)$ which consists of elements isotopic to the identity. 
\begin{definition}
 The mapping class group of $S$, denoted $\Mod(S),$ is defined to be the group of isotopy classes of orientation preserving diffeomorphisms of $S$:
 \begin{equation*}
  \Mod(S) = \text{Diff}^{+}(S) / \text{Diff}_{0}(S).
 \end{equation*}
\end{definition}
An element in $\Mod(S)$ is called a mapping class. The Dehn-Lickorish theorem shows that the mapping class group of an oriented surface is generated by Dehn twists. Let $A$ be the annulus $S^{1} \times [0, 1],$ which is represented by
\begin{equation*}
 \{(e^{i\,\theta}, \, t)\left| 0 \leqslant \theta \leqslant 2 \pi, \, 0 \leqslant t \leqslant 1 \right.\}.
\end{equation*}
Or equivalently, we represent $A$ by the set 
\[ \{(\theta, \, t) \left| 0 \leqslant \theta \leqslant 2 \pi, \, 0 \leqslant t \leqslant 1 \right. \}. \]
A twisting map $\psi : A \to A$ is defined as
\begin{equation*}
 (\theta, \, t) \mapsto (\theta + 2\pi t, \, t).
\end{equation*}
Let $\gamma$ be a simple loop in $S$, with a regular neighborhood $\mathscr{C}(\gamma) \subset S$. There exists an orientation preserving diffeomorphism $H : \mathscr{C}(\gamma) \to A.$ 
\begin{definition}
 The Dehn twist about $\gamma$, denoted $D_{\gamma}$, is defined to be the diffeomorphism $H^{-1} \circ \psi \circ H : \mathscr{C}(\gamma) \to \mathscr{C}(\gamma).$
\end{definition}
In the convention, we call $D_{\gamma}$ a positive Dehn twist, sometimes denoted $D_{\gamma}^{+}$, which is a left twist. And $D_{\gamma}^{-1}$ is called a negative Dehn twist, denoted $D_{\gamma}^{-}$.
Let $\Sigma_{g}$ be a closed orientable surface with genus $g.$ The following theorem is called Dehn-Lickorish theorem.
\begin{theorem}[Farb and Margalit {\cite[Theorem 4.1]{FaMa12}} ]
 The mapping class group $\Mod(\Sigma_{g})$ is generated by finitely many Dehn twists along nonseparating simple loops in $\Sigma_{g}.$
 \label{Dehn-Lickorish}
\end{theorem} 
\begin{remark}
 Proved by Humphries, the number of nonseparating simple loops in Dehn-Lickorish theorem can be taken as $2g + 1,$ see \cite{FaMa12}.
\end{remark}
\begin{remark}
 On a closed orientable hyperbolic surface $(\Sigma_g, \G)$, Luo proved that the number of Dehn twists in Dehn-Lickorish theorem is computable, see \cite{FrMeLu02}. 
 \begin{theorem}[Freedman, Meyer and Luo {\cite[Appendix, Theorem 12.9]{FrMeLu02}} ]
  There exists a computable constant $C(g, r)$ so that each isometry of $\Sigma_g$ is isotopic to a composition of positive and negative Dehn twists $D_{c_1}^{\pm 1} \circ D_{c_2}^{\pm 1} \circ \cdots \circ D_{c_k}^{\pm 1}$, where $k \leqslant C(g, r)$, and $r$ is the injectivity radius of $\Sigma_g$. Moreover, for each $i$, $\length_{\G}(c_i) \leqslant C(g, r)$.
 \label{Dehn-Lickorish-Luo}
 \end{theorem}
 When $r \geqslant \log{2}$, we have $C(g, r) \geqslant g^{g^{g^{\cdots g}}}$ (there are $3g - 3$ exponents).
\end{remark}

The elements of $\Mod(\Sigma_{g})$ are classified into three types: periodic, reducible and pseudo-Anosov, see \cite[13.3]{FaMa12}. A  class $f \in \Mod(\Sigma_{g})$ is periodic if it is of finite order. For a periodic mapping class $f \in \Mod(\Sigma_{g}),$ there exists a representative $\phi \in \text{Diff}^{+}(\Sigma_{g})$ so that $\phi$ is of finite order, see \cite[Theorem 7.1]{FaMa12}. 

Let $M$ be a smooth manifold, and let $\phi : M \to M$ be a diffeomorphism. 
\begin{definition}
 The mapping torus $M_{\phi}$ is defined to be a fiber bundle over the diffeomorphism map $\phi$ with fiber the manifold $M,$ which can be obtained from the cylinder $M \times [0, \, 1]$ by identifying the two ends via the map $\phi.$  
\end{definition} 
A $3$-manifold is called with surface bundle structure if it is a fiber bundle over the circle $S^{1}$ with fiber a closed orientable surface. A $3$-manifold $M$ with surface bundle structure is a mapping torus, i.e.,
\begin{equation*}
 M = \Sigma_{g} \times [0, \, 1] / (x, \, 0) \sim (\phi(x), \, 1),
\end{equation*} 
where $\Sigma_{g}$ is the fiber surface of $M,$ and the mapping class represented by $\phi: \Sigma_{g} \to \Sigma_{g}$ is called the monodromy of $M.$  According to Thurston, the geometric structure of $M$ is dependent on the type of the monodromy $\phi.$ Let $\mathbb{H}^{2}$ be the hyperbolic plane, and $\mathbb{R}$ be the Euclidean line. The following proposition is contained in Thurston's theorem.
\begin{proposition}[Farb and Margalit {\cite[Theorem 13.4]{FaMa12}}]
 Let $M$ be a $3$-manifold with surface bundle structure. Then $M$ has $\mathbb{H}^{2} \times \mathbb{R}$ geometric structure if and only if the monodromy of $M$ is periodic.
 \label{Thurston}
\end{proposition}  

At last of this section, we introduce Dehn surgery of $3$-manifolds. Let $M$ be a $3$-manifold, with boundary possibly. Let $K$ be a knot in the interior of $M.$ We use $\mathcal{T}(K)$ to denote a tubular neighborhood of $\gamma$ in the interior of $M,$ which is a solid torus. Let $\gamma$ be a loop in the boundary torus $\partial \mathcal{T}(K).$
\begin{definition}
 The Dehn surgery on $M$ with respect to the knot $K$ and the loop $\gamma$ is the following construction:
 \begin{equation*}
  M^{\prime} = \left( M - (\mathcal{T}(K))^{\circ} \right) \cup_{f} \tilde{\mathcal{T}}, 
 \end{equation*}
 where $\tilde{\mathcal{T}}$ is a solid torus, and $f: \partial \tilde{\mathcal{T}} \to \partial \mathcal{T}(K)$ is a homeomorphism on the boundary tori, such that the meridian loop of $\partial \tilde{\mathcal{T}}$ is mapped onto the loop $\gamma$ in $\partial \mathcal{T}(K).$
\end{definition}
It is called the Dehn filling of gluing a solid torus $\tilde{\mathcal{T}}$ to the $3$-manifold $M - (\mathcal{T}(K))^{\circ}$ with torus boundary. We can see \cite{Ro03} for more about Dehn surgery.

\section{$\Z_{2}$-coefficient homology $(1, 2)$-systolic freedom of $S^{2} \times S^{1}$}

Freedman \cite{Fr99} (see also \cite{FrMeLu02}) proved that the $3$-manifold $S^{2} \times S^{1}$  is of $\Z_{2}$-coefficient homology $(1, 2)$-systolic freedom. We introduce Freedman's technique in this section. Another interpretation of Freedman's work can be found in Fetaya \cite{Fe11}.

In Freedman's proof, a sequence of arithmetic hyperbolic surfaces $\{\Sigma_{g_{k}}\}_{k=1}^{\infty}$ is constructed first. In terms of the hyperbolic surface $\Sigma_{g_{k}}$, a Riemannian mapping torus $M_{g_{k}}$ with the metric $\G_{k}$ is constructed, where the Riemannian metric $\G_{k}$ is locally isometric to the standard product metric on $\mathbb{H}^{2} \times \mathbb{R}$. The mapping torus $M_{g_{k}}$ has fiber surface $\Sigma_{g_{k}}$. A new $3$-manifold homeomorphic to $S^{2} \times S^{1}$, denoted $S^{2} \times S^{1}_{g_{k}}$, is obtained by performing a series of Dehn surgeries on $M_{g_{k}}$. After specifying the metric change in Dehn surgeries, we have a Riemannian metric $\widehat{\mathcal{G}}_{k}$ on $S^{2} \times S^{1}_{g_{k}}.$ The metric change in Dehn surgeries is not given in Freedman's papers \cite{Fr99} and \cite{FrMeLu02}. We will describe such a metric change in Section 6 of this article. Based on topological and geometrical properties of $M_{g_{k}}$ as well as the metric change in Dehn surgeries, Freedman estimated the growth of $\sys H_{1}(S^{2} \times S^{1}, \widehat{\G}_{k}; \Z_{2})$ and $\sys H_{2}(S^{2} \times S^{1}, \widehat{\G}_{k}; \Z_{2})$ as well as the volume $\vol_{\widehat{\G}_{k}}(S^{2} \times S^{1})$. These estimations yield that
\begin{equation}
 \inf_{k} \ts \frac{\vol_{\widehat{\G}_{k}}(S^{2} \times S^{1})}{\sys H_{1}(S^{2} \times S^{1}, \widehat{\G}_{k}; \Z_{2}) \cdot \sys H_{2}(S^{2} \times S^{1}, \widehat{\G}_{k}; \Z_{2})} = 0.
 \label{Freedman}
\end{equation} 
Hence the $3$-manifold $S^{2} \times S^{1}$ is of $\Z_{2}$-coefficient homology $(1, 2)$-systolic freedom. We show more details of Freedman's technique in the following. 

Let $p$ be a prime number such that $p \equiv 3 \ts (\text{mod} \, 4).$ Define the group $\Gamma_{(-1, \, p)}$ as
\begin{equation*}
 \left\{ \left. \left( \begin{array}{ll} a + b\sqrt{p} & -c + d\sqrt{p} \\ c + d\sqrt{p} & a - b\sqrt{p} \end{array} \right) \right| a, b, c, d \in \mathbb{Z}, \, \det = 1 \right\} {\bigg /} \pm \text{I}_{2}, 
\end{equation*}
where $\det$ denotes the determinant of $2 \times 2$ matrix, and $I_2$ denotes the $2 \times 2$ identity matrix. The group $\Gamma_{(-1, \, p)}$ is an arithmetic Fuchsian group derived from the quaternion algebra
\begin{displaymath}
 \left( \frac{-1, \, p}{\mathbb{Q}} \right),
\end{displaymath}
see Schmutz Schaller \cite{Scha97}. 
Let $N \geqslant 2$ be a positive integer. Define the $N$-th congruence subgoup of $\Gamma_{(-1, \, p)}$ as
\begin{equation*}
  \left\{ \left. \left( \begin{array}{ll} 1 + N( a + b\sqrt{p} ) & N( -c + d\sqrt{p} ) \\ N( c + d\sqrt{p} ) & 1 + N( a - b\sqrt{p} ) \end{array} \right) \right| a, b, c, d \in \mathbb{Z}, \det = 1 \right\} {\bigg /} \pm \text{I}_{2},
\end{equation*}
denoted $\Gamma_{(-1, \, p)} (N)$. In terms of the congruence subgroup $\Gamma_{(-1, \ts p)}(N),$ we construct an arithmetic Riemann surface $\mathbb{H}^{2} / \Gamma_{(-1, \ts p)}(N)$, which has the following properties.
\begin{proposition}[Schmutz Schaller \cite{Scha97}]
 \begin{enumerate}
  \item The arithmetic Riemann surface $\mathbb{H}^{2}/\Gamma_{(-1, \, p)}(N)$ is hyperbolic.
  \item If the genus of $\mathbb{H}^{2} / \Gamma_{(-1, p)}(N)$ is denoted $\text{genus}\tts(N)$, then
        \begin{equation*}
         A_{p} \tts N^{2} \leqslant \text{genus}\tts(N) \leqslant B_{p} \tts N^{3}, 
        \end{equation*}
        where $A_{p}$ and $B_{p}$ are two positive constants which only depend on $p.$
  \item Let $\mathcal{G}_{\mathbb{H}^{2}}$ be the hyperbolic metric on $\mathbb{H}^{2}/\Gamma_{(-1, \, p)}(N)$. We have
        \begin{equation*}
         \sys \pi_{1}(\mathbb{H}^{2}/\Gamma_{(-1, \, p)}(N), \mathcal{G}_{\mathbb{H}^{2}}) \geqslant C_{1} \, \log{N},
        \end{equation*}
        where $C_{1}$ is a fixed positive constant independent of $N.$        
 \end{enumerate}
\end{proposition}
In the following, we use $\Sigma_{g}$ to denote the arithmetic hyperbolic surface $\mathbb{H}^{2} / \Gamma_{(-1, \, p)}(N)$ with genus $g.$ Proposition 4.1 implies that
\begin{equation}
 \sys \pi_{1}(\Sigma_{g}, \mathcal{G}_{\mathbb{H}^{2}}) \geqslant C_{2} \ts \log{g},
\end{equation}
where $C_{2}$ is a fixed positive constant independent of $g.$
Moreover, according to Buser and Sarnak \cite{BuSa94}, we have
\begin{equation}
 \lambda_{1}(\Sigma_{g}) \geqslant c_{1},
\end{equation}
where $\lambda_{1}$ is the first eigenvalue of the Laplacian on $\Sigma_{g},$ and $c_{1}$ is a fixed positive constant independent of $g.$ Let $A$ be a connected open subset of $\Sigma_{g}$, with area less than half of the area of $\Sigma_{g}$. By Buser's isoperimetric inequality (see \cite{Bu82}), we have
\begin{equation}
 \text{Area}_{\mathcal{G}_{\mathbb{H}^{2}}}(A) \leqslant C_{2} \, \length_{\G_{\mathbb{H}^{2}}} (\partial A) . 
 \label{isoperimetric}
\end{equation}
Let $p$ be a point on $\Sigma_{g}$. We take a disk $B_{t}(p)$ with radius $t$ and center $p$ on the hyperbolic surface $\Sigma_{g}$. Applying the polar coordinate of the hyperbolic metric $\G_{\mathbb{H}^{2}}$ on $\Sigma_{g}$, we have
\begin{equation*}
 \area (B_{t}(p)) = \int_{0}^{t} \ts \length_{\G_{\mathbb{H}^{2}}}(\partial B_{\rho}(p)) \, d\rho .
\end{equation*}
By (\ref{isoperimetric}), we have
\begin{align*}
 \frac{d}{dt} \, \left( \area ( B_{t}(p) ) \right) & = \length_{\G_{\mathbb{H}^{2}}}(\partial B_{t}(p)) \\
                                                       & \geqslant \frac{1}{C_{2}} \ts \area (B_{t}(p)).     
\end{align*}
On the other hand, we have $\area_{\G_{\mathbb{H}^{2}}}(\Sigma_{g}) = 4 \pi (g - 1)$ by the Gauss-Bonnet formula, so that
\begin{align*}
 t \leqslant C_{3} \ts \log{g}.
\end{align*} 
Hence we have
\begin{equation}
 \text{diameter}(\Sigma_{g}) \leqslant C_{3} \, \log{g}
 \label{diameter}
\end{equation}
with $C_{3}$ a positive constant independent of $g$, which yields that
\begin{equation*}
 \length_{\G_{\mathbb{H}^{2}}} ( \ell_{i} ) \leqslant C_{3} \, \log{g}, 
\end{equation*}
where $\{\ell_{1}, \, \ell_{2}, \, \cdots , \, \ell_{2g_{k}} \}$ is a system of loops representing a homology basis of $H_{1}(\Sigma_{g}; \Z)$.

Let $g$ be large enough. We construct an isometry map $\tau: \Sigma_{g} \to \Sigma_{g}$ of finite order (see \cite{Fr99}), such that
\begin{equation*}
 \text{Order}(\tau) \geqslant c_{2} \, ( \log{g} )^{1/2},
\end{equation*}
where $\text{Order}(\tau)$ stands for the order of $\tau,$ and $c_{2}$ is a fixed positive constant independent of $g.$ Therefore, we obtain a sequence of arithmetic hyperbolic surfaces $\{ \Sigma_{g_{k}} \}_{k=1}^{\infty}$ such that
\begin{equation*}
 2 \leqslant g_{1} < g_{2} < \cdots < g_{k} < \cdots  \qquad\quad \text{and } \qquad\quad \lim_{k\to \infty} g_{k} = \infty,
\end{equation*}
with a finite order isometry $\tau_{k}: \ts \Sigma_{g_{k}} \to \Sigma_{g_{k}}$ associated to each $\Sigma_{g_{k}}.$

Construct the mapping torus
\begin{equation*}
 M_{g_{k}} = \Sigma_{g_{k}} \times [0, 1] / (x, 0) \thicksim (\tau_{k}(x), 1),
\end{equation*}
where $\Sigma_{g_k}$ is the fiber surface. The monodromy represented by $\tau_{k}$ is of finite order. According to Thurston's theorem (Proposition \ref{Thurston}), $M_{g_{k}}$ has geometric structure $\mathbb{H}^{2} \times \R$. Then on $M_{g_{k}}$ we have a Riemannian metric $\mathcal{G}_{k}$ which is locally isometric to the standard product metric on $\mathbb{H}^{2} \times \mathbb{R}.$ Hence in terms of the sequence of arithmetic hyperbolic surfaces $\{\Sigma_{g_{k}}\}_{k=1}^{\infty},$ we have a sequence of Riemannian mapping tori $\{(M_{g_{k}}, \mathcal{G}_{k})\}_{k=1}^{\infty}.$ 

As the isometry map $\tau_{k}$ is of finite order, by Dehn-Lickorish theorem (see Theorem \ref{Dehn-Lickorish}), we have
\begin{equation*}
 \tau_{k} = \sigma_{1} \circ \sigma_{2} \circ \cdots \circ \sigma_{n_{k}},
\end{equation*}
where $\sigma_{1}, \sigma_{2}, \cdots , \sigma_{n_{k}}$ are positive or negative Dehn twists along simple geodesic loops $\gamma_1, \gamma_2, \cdots , \gamma_{n_k}$ in $\Sigma_{g_{k}}.$ The injectivity radius of $\Sigma_{g_k}$ is $\frac{\sys \pi_{1}(\Sigma_{g_{k}}, \G_{\mathbb{H}^2})}{2}$. According to Theorem \ref{Dehn-Lickorish-Luo}, there exists a positive constant $C(g_k)$ only dependent on $g_k$, such that $n_k \leqslant C(g_{k})$ and $\length(\gamma_{i}) \leqslant C(g_k)$ for each $i$. For each Dehn twist $\sigma_{i}$, we perform a Dehn surgery in $M_{g_{k}}$. All Dehn surgeries are performed at different fiber surface levels in the interval direction of $M_{g_{k}}$.  Let the radius $\varepsilon_{k}$ of solid tori in Dehn surgeries be small enough, which is controlled by $C(g_{k})$. Then all Dehn surgeries are mutually disjoint. We have a mapping torus with the monodromy represented by $\tau_{k}^{-1} \circ \tau_{k},$ which is the identity. Hence we obtain a mapping torus homeomorphic to $\Sigma_{g_{k}} \times S^{1}$. 

The arithmetic hyperbolic surface $\Sigma_{g_{k}}$ has a system of $2g_{k}$ nonseparating geodesic loops representing a basis of $H_{1}(\Sigma_{g_{k}}; \Z)$. Around each geodesic loop, we perform a Dehn surgery to let the geodesic loop be contractible. When the radius $\varepsilon_k$ is small enough, all of these $2g_k$ Dehn surgeries are mutually disjoint. Finally we obtain a $3$-manifold homeomorphic to $S^{2} \times S^{1},$ denoted $S^{2} \times S^{1}_{g_{k}}$. We use the cutoff function technique to get a smooth Riemannian metric after each Dehn surgey, see Section 6. After all $n_{k}+2g_{k}$ Dehn surgeries, we obtain a smooth Riemannian metric $\hat{\G}_{k}$ on $S^{2} \times S^{1}_{g_{k}}.$ Hence we get a sequence of Riemannian metrics on $S^{2} \times S^{1}$, denoted by
\begin{equation*}
 \{ (S^{2} \times S^{1}_{g_{k}}, \, \hat{\mathcal{G}}_{k}) \}_{k = 1}^{\infty}.
\end{equation*}    

In terms of systolic propositions of the arithmetic hyperbolic surface $\Sigma_{g_k}$ and the homological property of $S^2\times S^1$, as well as geometric properties of Dehn surgeries, the following systolic bound estimations can be derived on $S^{2} \times S^{1}_{g_{k}}$.
\begin{theorem}[Freedman \cite{Fr99}]
 There exist positive constants $c_4, c_5$ and $c_6$ independent of $g_{k}$ such that
 \begin{enumerate}
  \item
   \begin{equation}
    \sys H_{1}(S^{2} \times S^{1}_{g_{k}}, \widehat{\mathcal{G}}_{k}; \mathbb{Z}_{2}) \geqslant c_{4} \ts (\log{g_{k}})^{1/2};
   \end{equation}
  \item
   \begin{equation}
    \sys H_{2}(S^{2} \times S^{1}_{g_{k}}, \widehat{\G}_{k}; \Z_{2}) \geqslant c_{5} \ts g_{k};
   \end{equation}
  \item
   \begin{equation}
    \vol_{\widehat{\G}_{k}} (S^{2} \times S^{1}) \leqslant c_{6} g_{k}.
   \end{equation}
 \end{enumerate}
 \label{Fre_estimation}
\end{theorem}
Based on above esitmations in Theorem \ref{Fre_estimation}, we obtain (\ref{Freedman}). Hence the $3$-manifold $S^2 \times S^1$ is of $\Z_{2}$-coefficient homology $(1, 2)$-systolic freedom.   

\section{$3$-manifolds with semibundle structure}

Roughly speaking, a $3$-manifold with the semibundle structure is composed by the union of two twisted $I$-bundles, which are glued together along their common boundary surfaces. We introduce some preliminary knowledge of semibundles in this section, more details can be found in \cite{He04, HeJa72, Zu97}.

Let $M$ be a closed and connected $3$-manifold. A halving of $M$ is an index two subgroup of $\pi_{1}(M).$ When $M$ has a halving $H,$ there exists a two-sheeted covering $3$-manifold whose fundamental group is isomorphic to $H$, denoted $M_{H}.$ Assume that $Q_{H}: M_{H} \to M$ is the covering map, and $\alpha_{H}: M_{H} \to M_{H}$ is the covering translation. Then we have $M \simeq M_{H} / \alpha_{H}.$

We express the unit circle $S^{1}$ as the set of complex numbers with unit module on the complex plane $\C^{1}$. Denote the closed interval $[-1, 1]$ by $D^{1}$. Let $\tau: S^{1} \to S^{1}$ be the map of complex conjugation, and $q: S^{1} \to D^{1}$ be the projection map defined by $z \mapsto \text{Re}(z).$ The semibundle structure of $3$-manifolds can be defined in the following way, see Zulli \cite{Zu97}.
\begin{definition}
 Let $M$ be a closed and oriented connected $3$-manifold with the halving $H$. A map $f: M \to D^{1}$ is called a semibundle subordinate to the halving $H$ (or an $H$-semibundle) with regular fiber surface $\Sigma$ if there exists a two-sheeted covering surface bundle $F: M_{H} \to S^{1}$ with fiber surface $\Sigma,$ such that $q \circ F = f \circ Q_{H}$ and $F \circ \alpha_{H} = \tau \circ F.$ 
 \label{semi}
\end{definition}

The definition yields the following properties of an $H$-semibundle.
\begin{proposition}[Zulli \cite{Zu97}]
 Let $f: M \to D^{1}$ be an $H$-semibundle. Let $\Sigma_{g}$ be a closed surface with genus $g.$ Assume that the covering surface bundle $M_{H}$ has fiber surface $\Sigma_{g}.$
 \begin{enumerate}
  \item When $t\in (-1, \, 1),$ $f^{-1}(t)$ is homeomorphic to the fiber surface $\Sigma_{g}$ of $M_{H}.$ The regular fiber surface $f^{-1}(t)$ is lifted to two copies of $\Sigma_{g}$ in $M_{H}$ which are exchanged by $\alpha_{H}.$ When $t = -1$ or $t = 1,$ $f^{-1}(t)$ is an embedded surface which is doubly covered by the fiber surface $F^{-1}(t) = \Sigma_{g}.$
  \item If we use $J$ to denote the interval $[-1, \, 0]$ or $[0, \, 1],$ then $f^{-1}(J)$ is a twisted $I$-bundle. Denote two twisted $I$-bundles $f^{-1}([-1, \, 0])$ and $f^{-1}([0, \, 1])$ by $M_{1}$ and $M_{2}$ respectively. We have $M = M_{1} \cup M_{2},$ and $M_{1} \cap M_{2} = \Sigma_{g}.$ 
 \end{enumerate}
\end{proposition}

We have the following estimation for homotopy $1$-systole of an $H$-semibundle $M.$ 
\begin{proposition}
 Let $M$ be an $H$-semibundle with the Riemannian metric $\mathcal{G}.$ We use $\tilde{\G}$ to denote the covering metric on covering surface bundle $M_{H}$. Then we have
 \begin{equation}
  \sys \pi_{1}(M, \mathcal{G}) \geqslant \frac{1}{2} \, \sys \pi_{1}(M_{H}, \, \tilde{\mathcal{G}}).
  \label{semi_1}
 \end{equation}
 \label{semibundle_1}
\end{proposition}

\begin{proof}
For every noncontractible loop $\gamma$ in $M,$ there exists a lifting to the covering surface bundle $M_{H}.$ We denote the lifting of $\gamma$ in $M_{H}$ by $\tilde{\gamma}.$ Suppose that $[\gamma]$ is the homotopy class in $\pi_{1}(M)$ represented by $\gamma,$ and $[\tilde{\gamma}]$ is the homotopy class in $\pi_{1}(M_{H})$ represented by $\tilde{\gamma}.$
 
If $[\gamma] \in H,$ as $(Q_{H})_{*}: \pi_{1}(M_{H}) \to \pi_{1}(M)$ is injective and 
\begin{equation*}
 (Q_{H})_{*}(\pi_{1}(M_{H})) = H ,
\end{equation*}
we must have $[\tilde{\gamma}] \neq 1$ in $\pi_{1}(M_{H}).$ Hence we have
\begin{equation*}
 \length_{\G}(\gamma) = \length_{\tilde{\G}}(\tilde{\gamma}).
\end{equation*}
If $[\gamma] \in \pi_{1}(M) - H,$ as $\pi_{1}(M) / H \simeq \mathbb{Z}_{2}$, we have $\pi_{1}(M) = H \cup [\gamma]H$ so that $[\gamma]^{2} \in H.$ Hence the union $\hat{\gamma} = \tilde{\gamma} \cup \alpha_{H}(\tilde{\gamma}) $ is a noncontractible loop in $M_{H}$, which doubly covers $\gamma.$ Then we have
\begin{equation*}
 \length_{\G}(\gamma) = \frac{1}{2} \ts \length_{\tilde{\G}}(\hat{\gamma}). 
\end{equation*}
After taking the infimum over all noncontractible loops $\gamma$ in $M,$ we have
\begin{equation*}
 \sys \pi_{1}(M, \G) \geqslant \frac{1}{2} \tts \sys \pi_{1}(M_{H}, \tilde{\G}).
\end{equation*}
\end{proof}

\section{Proof of the main theorem}

The proof of Theorem \ref{Chen14} is separated into two steps. In the first step, we apply Freedman's technique (see Section 4) to construct a sequence of Riemannian metrics on $\RP^{3} \# \RP^{3}$, denoted by
\begin{equation*} 
 \{ (\RP^{3} \# \RP^{3}_{g_{k}}, \, \hat{\mathscr{G}}_{k}) \}_{k = 1}^{\infty} .
\end{equation*} 
Then in the second step, we prove that the sequence of metrics $\hat{\mathscr{G}}_{k}$ yields $\Z_{2}$-coefficient homology $(1, 2)$-systolic freedom.

\subsection{Construction of Riemannian metrics on $\RP^{3} \# \RP^{3}$}

In Section 4, we introduced Freedman's proof of $\Z_2$-coefficient homology $(1, 2)$-systolic freedom on $S^{2} \times S^1$. In Freedman's work, a sequence of arithmetic hyperbolic surfaces
\begin{equation*}
 \{ (\Sigma_{g_{k}}, \, \tau_{k}) \}_{k = 1}^{\infty},
\end{equation*}
is constructed, where $\tau_{k}: \, \Sigma_{g_{k}} \to \Sigma_{g_{k}}$ is the isometry map of finite order. Upon $(\Sigma_{g_{k}}, \tau_k)$, the Riemannian mapping torus $(M_{g_{k}}, \G_k)$ is constructed. 

In this section, we use $I_{1}$ to denote the interval $[-1, \, 0],$ and use $I_{2}$ to denote the interval $[0, \, 1].$ Define two twisted $I$-bundles over $I_{1}$ and $I_{2}$ with the regular fiber surface $\Sigma_{g_{k}},$ denoted $ \Sigma_{g_{k}} \tilde{\times} I_{1}$ and $ \Sigma_{g_{k}} \tilde{\times} I_{2}$ respectively. For each $k,$ we construct a $3$-manifold $N_{g_{k}}$ with the following semibundle structure: 
\begin{equation*}
 \left( \Sigma_{g_{k}} \tilde{\times} I_{1} \right) \cup_{\tau_{k}} \left( \Sigma_{g_{k}} \tilde{\times} I_{2} \right),
\end{equation*} 
where the two twisted $I$-bundles are glued together along their common boundary surface $\Sigma_{g_{k}}.$ The $3$-manifold $N_{g_{k}}$ is doubly covered by the mapping torus $M_{g_{k}}$. Hence there is a Riemannian metric on $N_{g_{k}}$, which is locally isometric to the standard product metric on $\mathbb{H}^{2} \times \R,$ denoted $\mathscr{G}_{k}.$ The Riemannian metric $\G_{k}$ on $M_{g_{k}}$ becomes the covering metric. Denote $\pi_{1}(M_{g_k})$ by $H$. The $3$-manifold $N_{g_{k}}$ is an $H$-semibundle. Therefore in terms of $\{\Sigma_{g_k}\}_{k = 1}^{\infty}$, we constructed a sequence of Riemannian $H$-semibundles $\{(N_{g_{k}}, \mathscr{G}_{k})\}_{k = 1}^{\infty}.$ 

As mentioned in Section 4, the Dehn-Lickorish twist theorem implies that the isometry map $\tau_{k}$ is decomposed into a sequence of $n_{k}$ Dehn twists along $2g_{k} + 1$ nonseparating simple geodesic loops on $\Sigma_{g_{k}}.$ For each Dehn twist, we perform a Dehn surgery. Moreover, all Dehn surgeries are performed at different levels of the regular fiber surface $\Sigma_{g_{k}}.$ Hence they are pairwise disjoint when the radius $\varepsilon_{k}$ of solid tori is small enough. After we finish all these Dehn surgeries, the $3$-manifold $N_{g_{k}}^{\prime}$ obtained is homeomorphic to a semibundle which is doubly covered by the surface bundle $\Sigma_{g_{k}} \times S^{1}.$ The metric change in Dehn surgeries is described as follows.

Suppose that $\tau_{k} = \sigma_{1} \circ \sigma_{2} \circ \cdots \circ \sigma_{n_{k}}$ (see Section 4), where $\sigma_{1}, \sigma_{2}, \cdots , \sigma_{n_{k}}$ are Dehn twists along $2g_{k} + 1$ nonseparating simple geodesic loops on $\Sigma_{g_{k}}.$ Moreover, we assume that the Dehn twist $\sigma_{i}$ is along a nonseparating geodesic loop $\gamma_{i},$ for $i = 1, 2, \cdots , n_{k}.$ All Dehn surgeries in $N_{g_{k}}$ for these Dehn twists are performed at the following distinct regular fiber surface levels:
\begin{equation*}
 \gamma_{1} \times \left\{ \frac{1}{2} \right\}, \, \, \gamma_{2} \times \left\{ \frac{1}{2} + \frac{1}{3n_{k}} \right\}, \, \cdots , \, \gamma_{n_{k}} \times \left\{ \frac{1}{2} + \frac{n_{k} - 1}{3n_{k}} \right\}.
\end{equation*}
Let $L_{i, k}$ denote $\length_{\G_{\mathbb{H}^{2}}}(\gamma_{i})$. We have
\[ L_{i, k} \leqslant C(g_{k}) \]
according to Theorem \ref{Dehn-Lickorish-Luo}. Let $\mathscr{D}_{i}$ be the Dehn surgery corresponding to the twist $\sigma_{i}$. In $\mathscr{D}_{i}$, we first remove a solid torus $T_{i, \ts \varepsilon_{k}}$ with the radius $\varepsilon_{k}$. The solid torus $T_{i, \ts \varepsilon_{k}}$ is a tubular neighborhood of a geodesic loop in $N_{g_{k}}.$ And $\gamma_{i}$ is the longitude loop on boundary torus $\partial T_{i, \, \varepsilon_{k}}.$ The radius $\varepsilon_k$ is assumed to be small enough, which is less than $1/C(g_k)$. Let $\delta_{k} = \varepsilon_{k}/4$. We fill in a solid torus $\tilde{T}_{i, \, \varepsilon_{k} + \delta_{k}}$ with the radius $\varepsilon_{k} + \delta_{k}$ to the semibundle complement $N_{g_{k}} - T_{i, \ts \varepsilon_{k}}^{\circ}.$ The filled in solid torus $\tilde{T}_{i, \, \varepsilon_{k} + \delta_{k}}$ is defined as follows. Let $\bar{T}_{i, \, \varepsilon_{k} + \delta_{k}}$ be the solid torus 
\begin{equation*}
 \left\{ (r, \theta, t) \left| 0 \leqslant r \leqslant \frac{L_{i, \, k}}{2 \pi} + \delta_{k}, \, 0 \leqslant \theta \leqslant 2\pi, \, 0 \leqslant t \leqslant 2\pi \varepsilon_{k} \right. \right\} {\bigg /} \sim ,
\end{equation*}
with $\sim$ the identification $(r, \, \theta, \, 0) \mapsto (r, \, \theta, \, 2\pi\varepsilon_{k}).$ On $\bar{T}_{i, \, \varepsilon_{k} + \delta_{k}},$ we define the Euclidean metric $\bar{\mathcal{G}}_{k},$ which is expressed as $dr^{2} + r^{2}d\theta^{2} + dt^{2}$. Define a smooth twisting map $\beta_{i, \, k}$ on the solid torus $\bar{T}_{i, \, \varepsilon_{k} + \delta_{k}}.$ When restricted to the disk
\begin{equation*}
 \{(r, \, \theta, \, \pi\varepsilon_{k}) \left| 0 \leqslant r \leqslant \frac{L_{i, \, k}}{2\pi} + \delta_k, \, 0 \leqslant \theta \leqslant 2\pi \right. \}
\end{equation*}
in $\bar{T}_{i, \ts \varepsilon_k + \delta_k}$, the twisting map $\beta_{i, \, k}$ is a $\pi$-rotation. And outside of a neighborhood of this disk, it is the identity map. After performing the twisting $\beta_{i, \, k}$ on $\bar{T}_{i, \, \varepsilon_{k} + \delta_{k}},$ we have the desired filled in solid torus $\tilde{T}_{i, \, \varepsilon_{k} + \delta_{k}} = \beta_{i, \, k}(\bar{T}_{i, \, \varepsilon_{k} + \delta_{k}}).$ The metric on $\tilde{T}_{i, \, \varepsilon_{k} + \delta_{k}}$ is defined to be the pullback metric $(\beta_{i, \, k}^{-1})^{*} \, \bar{\mathcal{G}}_{k},$ denoted $\tilde{\G}_{k}.$ 

Next we fill in the solid torus $\tilde{T}_{i, \, \varepsilon_{k} + \delta_{k}}$ to the semibundle complement $N_{g_{k}} - T_{i, \, \varepsilon_{k}}^{\circ}.$ Let $\bar{Y}_{i, \, k}$ be an annulus product in the solid torus $\bar{T}_{i, \, \varepsilon_{k}+\delta_{k}}$ defined by
\begin{equation*}
\left\{ (r, \, \theta, \, t) \left| \frac{L_{i, \, k}}{2\pi} \leqslant r \leqslant \frac{L_{i, \, k}}{2\pi} + \delta_{k}, \, 0 \leqslant \theta \leqslant 2\pi , \, 0 \leqslant t \leqslant 2 \pi \varepsilon_{k} \right. \right\} {\bigg /} \sim .
\end{equation*}
Hence we have an annulus product $\tilde{Y}_{i, \, k} = \beta_{i, \, k}(\bar{Y}_{i, \, k}) \subset \tilde{T}_{i, \tts \varepsilon_k + \delta_k}$. Let $T_{i, \ts \varepsilon_k + \delta_k}$ be the solid torus with radius $\varepsilon_k + \delta_k$. The solid torus $T_{i, \ts \varepsilon_k + \delta_k}$ has the same core loop with $T_{i, \ts \varepsilon_k}$, and $T_{i, \ts \varepsilon_k} \subset T_{i, \ts \varepsilon_k + \delta_k}$. Define the annulus product $Y_{i, \ts k} \subset T_{i, \ts \varepsilon_k + \delta_k}$ by
\begin{equation*}
 \left\{ (r, \, \theta, \, t) \left| \, \varepsilon_{k} \leqslant r \leqslant \varepsilon_{k} + \delta_{k}, \, 0 \leqslant \theta \leqslant 2 \pi, \, 0 \leqslant t \leqslant L_{i, \, k} \right. \right\} / \sim,
\end{equation*}
so that $T_{i, \ts \varepsilon_k + \delta_k} = T_{i, \ts \varepsilon_k} \cup Y_{i, \ts k}$. The gluing map $f_{i, \, k}: \tilde{Y}_{i, \, k} \to Y_{i, \, k}$ is defined as
\begin{equation*}
 (r, \, \theta, \, t) \mapsto \left( r - \frac{L_{i, \, k}}{2\pi} + \varepsilon_{k}, \, \frac{t}{\varepsilon_{k}}, \, \frac{L_{i, \, k}}{2\pi} \theta \right).
\end{equation*}
We use $\bar{T}_{i, \, \varepsilon_{k}}$ to denote the solid torus with radius $L_{i, \, k}/2\pi$ inside of $\bar{T}_{i, \, \varepsilon_{k} + \delta_{k}}$, and use $\tilde{T}_{i, \, \varepsilon_{k}}$ to denote $\beta_{i, \, k}(\bar{T}_{i, \, \varepsilon_{k}})$. The solid torus $\bar{T}_{i, \ts \varepsilon_k}$ has the same core loop with $\bar{T}_{i, \ts \varepsilon_k + \delta_k}$, and $\bar{T}_{i, \ts \varepsilon_k + \delta_k} = \bar{T}_{i, \ts \varepsilon_k} \cup \bar{Y}_{i, \ts k}. $ If $\bar{m}_{i, \, k}$ stands for the meridian loop 
\begin{equation*}
 \left\{ \left. \left( \frac{L_{i, \, k}}{2\pi}, \, \theta , \, \pi \varepsilon_{k} \right) \right| 0 \leqslant \theta \leqslant 2\pi \right\}
\end{equation*}
of the boundary torus $\partial \bar{T}_{i, \, \varepsilon_{k}},$ then $\tilde{m}_{i, \, k} = \beta_{i, \, k}(\bar{m}_{i, \, k})$ is the meridian loop on the boundary torus $\partial \tilde{T}_{i, \, \varepsilon_{k}}.$ After gluing by using the filling map $f_{i, \, k}$, we have $\partial T_{i, \, \varepsilon_{k}} = f_{i, \, k}(\partial \tilde{T}_{i, \, \varepsilon_{k}}).$ The meridian loop $\tilde{m}_{i, \ts k}$ is glued with the longitude loop $\gamma_{i}$ on the boundary torus $\partial T_{i, \, \varepsilon_{k}}.$

In order to get a smooth Riemannian metric after the Dehn surgery $\mathscr{D}_{i},$ we apply the cutoff function technique. Define a smooth cutoff fucntion $\phi_{i, \ts k}$ on $\overline{N}_{g_{k}} = (N_{g_{k}} - T_{i, \, \varepsilon_{k}}^{\circ}) \cup_{f_{i, \, k}} \tilde{T}_{i, \, \varepsilon_{k} + \delta_{k}}$ as follows:
\begin{equation*}
 \phi_{i, \, k}(x) = \left\{
  \begin{array}{ll}
   0, & \text{if} \,\, x \in N_{g_{k}} - (T_{i, \, \varepsilon_{k} } \cup Y_{i, \, k} )^{\circ}; \\
    & \\
   1, & \text{if} \,\, x \in \tilde{T}_{i, \, \varepsilon_{k}}.
  \end{array}
 \right.
\end{equation*}
Moreover, we let $\phi_{i, \ts k}(x) \in (0, \, 1)$ for $x \in Y_{i, \, k}^{\circ}.$ And for $(\varepsilon_{k} + \delta_{k}/2, \, \theta, \, t) \in Y_{i, \, k}^{\circ},$ with $ 0 \leqslant \theta \leqslant 2\pi, \, 0 \leqslant t \leqslant L_{i, \, k},$ the value of $\phi_{i, \, k}$ is assumed to be equal to $1/2.$ In terms of the cutoff function $\phi_{i, \ts k}$, we define a smooth Riemannian metric $\hat{\mathscr{G}}_{k}$ on $\overline{N}_{g_{k}}$ to be 
\begin{equation*}
 \hat{\mathscr{G}}_{k} = \left\{
 \begin{array}{ll}
  \mathscr{G}_{k} & \text{if} \, \, x \in N_{g_{k}} - (T_{i, \, \varepsilon_{k}} \cup Y_{i, \, k})^{\circ}; \\
   & \\
  \phi_{i, \, k} \cdot (f_{i, \, k}^{-1})^{*} \tilde{\mathcal{G}}_{k} + (1 - \phi_{i, \, k}) \, \mathscr{G}_{k} & \text{if} \, \, x \in Y_{i, \, k}^{\circ} = f_{i, \, k}\left( \tilde{Y}_{i, \ts k}^{\circ} \right); \\
   & \\
  \tilde{\mathcal{G}}_{k} & \text{if} \, \,  x \in \tilde{T}_{i, \, \varepsilon_{k}} .
 \end{array} 
 \right.
\end{equation*}
For each $i = 1, \, 2, \, \cdots \, , \, n_{k}$, we use the cutoff function $\phi_{i, \, k}$ in the Dehn surgery $\mathscr{D}_{i}$ to get a smooth Riemannian metric. After these $n_{k}$ Dehn surgeries, we have a smooth Riemannian metric $\mathscr{\G}_{k}$ defined on $N_{g_{k}}^{\prime}$. 

In addition, the fiber surface $\Sigma_{g_{k}}$ has a set of $2g_{k}$ nonseparating simple geodesic loops $\{\ell_{1}, \, \ell_{2}, \, \cdots \, , \, \ell_{2g_{k}}\}$ representing the basis of $H_{1}(\Sigma_{g_{k}}; \, \mathbb{Z})$. We perform a Dehn surgery around each geodesic loop $\ell_{j},$ with $j = 1, \, 2, \, \cdots \, , \, 2g_{k}$. All of these $2g_{k}$ Dehn surgeries in $N_{g_{k}}^{\prime}$ are performed at the following different regular fiber surface levels:
\begin{equation*}
 \ell_{1} \times \left\{ \frac{5}{6} \right\}, \, \, \ell_{2} \times \left\{ \frac{5}{6} + \frac{1}{12g_{k}} \right\} , \,\, \cdots , \, \, \ell_{2g_{k}} \times \left\{ \frac{5}{6} + \frac{2g_{k} - 1}{12g_{k}} \right\} .
\end{equation*}
Hence together with Dehn surgeries $\mathscr{D}_{i}$ corresponding to Dehn twists $\sigma_{i},$ all Dehn surgeries are pairwise disjoint when the radius $\varepsilon_{k}$ of removed out solid tori is small enough. Compared with Dehn surgeries $\mathscr{D}_{i}$ corresponding to Dehn twists $\sigma_{i}$, the difference of Dehn surgeries here is that it is not needed to twist the filled in solid tori. Other steps are the same. After each Dehn surgery, the cutoff function technique is used to obtain a smooth Riemannian metric. Upon finishing these additional $2g_{k}$ Dehn surgeries, we obtain a $3$-manifold $N_{g_{k}}^{\prime\prime}$ whose double covering manifold is homeomorphic to $S^{2} \times S^{1}.$ Hence we have $\pi_{1}(N_{g_{k}}^{\prime\prime}) = \mathbb{Z}_{2} * \mathbb{Z}_{2}$, so that $N_{g_{k}}^{\prime\prime}$ is homeomorphic to $\RP^{3} \# \RP^{3}.$ We use $\RP^{3}\# \RP^{3}_{g_{k}}$ to denote $N_{g_{k}}^{\prime\prime}$ in the following. The smooth Riemannian metric on $\RP^{3} \# \RP^{3}_{g_{k}}$ obtained through the cutoff function technique is still denoted $\hat{\mathscr{G}}_{k}$ without any confusion. Therefore we have a sequence of Riemannian $3$-manifolds $\{( \RP^{3}\# \RP^{3}_{g_{k}}, \ts \hat{\mathscr{G}}_{k} ) \}_{k=1}^{\infty}.$ 
  
\subsection{$\Z_{2}$-coefficient homology $(1, 2)$-systolic freedom}

We show that under the sequence of Riemannian metrics $\{\hat{\mathscr{G}}_{k}\}_{k=1}^{\infty}$,
\begin{equation}
 \inf_{k} \ts \frac{\vol_{\hat{\mathscr{G}}_{k}}(\RP^{3}\# \RP^{3}_{g_{k}})}{\sys H_{1}(\RP^{3} \# \RP^{3}_{g_{k}}, \hat{\mathscr{G}}_{k}; \Z_{2}) \cdot \sys H_{2}(\RP^{3} \# \RP^{3}_{g_{k}}, \hat{\mathscr{G}}_{k}; \Z_{2}) } = 0.
 \label{semibundle} 
\end{equation}
The equality (\ref{semibundle}) implies that $\RP^{3} \# \RP^{3}$ is of $\Z_{2}$-coefficient homology $(1, 2)$-systolic freedom.

We prove (\ref{semibundle}) by establishing estimates similar to those in Theorem \ref{Fre_estimation}. For $\Z_{2}$-coefficient homology $1$-systole, we have the following lower bound estimation.
\begin{proposition}
 When $k$ is large enough, the $\mathbb{Z}_{2}$-coefficient homology $1$-systole of $( \RP^{3} \# \RP^{3}_{g_{k}}, \, \hat{\mathscr{G}}_{k})$ satisfies
 \begin{equation}
  \sys H_{1}(\RP^{3} \# \RP^{3}_{g_{k}}, \hat{\mathscr{G}}_{k}; \mathbb{Z}_{2}) \geqslant s_{1} \tts ( \log{g_{k}} )^{1/2},
  \label{RP3_1}
 \end{equation}
 where $s_{1}$ is a positive constant independent of $g_{k}.$
\end{proposition}

\begin{proof}
Let $\gamma$ be a noncontractible loop in $\RP^{3} \# \RP^{3}_{g_{k}}.$ If $\gamma$ intersects with a solid torus $\tilde{T}_{i, \, \varepsilon_{k} + \delta_{k}}$ in Dehn surgeries, we use a geodesic arc on the boundary torus $\partial \tilde{T}_{i, \, \varepsilon_{k} + \delta_{k}}$ to substitute the arc inside of the solid torus. The substitution geodesic arc on the boundary torus is composed with a geodesic arc parallel to the longitude geodesic loop and another geodesic arc parallel to the meridian geodesic loop. After such a substitution, the length increase will not exceed the length of meridian geodesic loop of the boundary torus. When we finish all possible substitutions by considering all Dehn surgeries, we have a new loop $\gamma^{\prime} \subset \RP^{3} \# \RP^{3}_{g_{k}}.$ Assume that $k$ is large enough. The radius $\varepsilon_{k}$ of solid tori would be small enough. So that $\gamma^{\prime}$ is homotopic to $\gamma$, which is noncontractible. We have the following estimation for the length difference,
\begin{align*}
 \left| \length_{\hat{\mathscr{G}}_{k}}(\gamma) - \length_{\hat{\mathscr{G}}_{k}}(\gamma^{\prime}) \right| & \leqslant 2 \pi (\varepsilon_{k} + \delta_{k}) \ts (n_k + 2g_{k} ) \\
      & \leqslant \frac{5\pi}{2} (C(g_k) + 2g_k) \ts \varepsilon_{k},
\end{align*}  
where the last inequality holds because of Theorem \ref{Dehn-Lickorish-Luo}, also see Freedman's example in Section 4. Assume that $\varepsilon_{k} = \frac{1}{g_k \cdot C(g_k)}$. When $k$ is large enough, there exists a constant $C$ independent of $g_{k}$ (for example, we can let $C = 1$) such that
\begin{align*}
 \length_{\hat{\mathscr{G}}_{k}}(\gamma) \geqslant \length_{\hat{\mathscr{G}}_{k}}(\gamma^{\prime}) - C .
\end{align*} 
As $\gamma^{\prime}$ has no intersection with solid tori of Dehn surgeries, it is a noncontractible loop in the semibundle $N_{g_{k}}$ after we do reverse Dehn surgeries on $\RP^{3} \# \RP^{3}_{g_{k}}.$ Therefore we have
\begin{align*}
 \length_{\hat{\mathscr{G}}_{k}}(\gamma) & \geqslant \length_{\hat{\mathscr{G}}_{k}}(\gamma^{\prime}) - C \\
                                         & \geqslant \sys \pi_{1}(N_{g_{k}}, \mathscr{G}_{k}) - C \\
                                         & \geqslant \frac{1}{2} \, \sys \pi_{1}(M_{g_{k}}, \mathcal{G}_{k}) - C, 
\end{align*}
where the last inequality holds because of inequality (\ref{semi_1}) in Proposition \ref{semibundle_1}. The proof of Proposition 2.3 in Freedman \cite{Fr99} in fact implies that 
\begin{equation*}
 \sys \pi_{1}(M_{g_{k}}, \, \mathcal{G}_{k}) \geqslant C^{\prime} \, (\log{g_{k}})^{1/2}, 
\end{equation*}
where $C^{\prime}$ is a constant independent of $g_{k}.$ Hence we have
\begin{equation*}
 \length_{\hat{\mathscr{G}}_{k}}(\gamma)  \geqslant \frac{1}{2} C^{\prime} \, (\log{g_{k}})^{1/2} - C.
\end{equation*}
After taking the infimum over all noncontractible loops $\gamma$ in $\RP^{3} \# \RP^{3}_{g_{k}}$, we have
\begin{equation*}
 \sys \pi_{1}(\RP^{3} \# \RP^{3}_{g_{k}}, \, \hat{\mathscr{G}_{k}})  \geqslant s_{1} \ts (\log{g_{k}})^{1/2},
\end{equation*}
where $s_{1}$ is a positive constant independent of $g_{k}.$ The inequality (\ref{RP3_1}) is implied by the following property of systoles:
\begin{equation*}
 \sys H_{1}(\RP^{3} \# \RP^{3}, \hat{\mathscr{G}}_{k}; \Z_{2}) \geqslant \sys \pi_{1}(\RP^{3} \# \RP^{3}_{g_{k}}, \hat{\mathscr{G}}_{k}).
\end{equation*}
\end{proof}

For $\Z_{2}$-coefficent homology $2$-systole, we have the following lower bound estimation.
\begin{proposition}
 When $k$ is large enough, the $\mathbb{Z}_{2}$-coefficient homology $2$-systole of $( \RP^{3} \# \RP^{3}_{g_{k}}, \hat{\mathscr{G}}_{k} )$ satisfies 
 \begin{equation}
  \sys H_{2}(\RP^{3} \# \RP^{3}_{g_{k}}, \hat{\mathscr{G}}_{k}; \Z_{2}) \geqslant s_{2} \ts g_{k},
  \label{RP3_2}
 \end{equation}
 where $s_{2}$ is a positive constant independent of $g_{k}$.
\end{proposition}

\begin{proof}
Assume that $X_{k}$ is a smooth embedded surface in $\RP^{3} \# \RP^{3}_{g_{k}}$, which is area minimizing among all surfaces representing a nonzero homology class in $H_{2}(\RP^{3} \# \RP^{3}_{g_{k}}; \Z_{2})$. The existence of $X_{k}$ is guaranteed by the geometric measure theory. We assume that $( (S^{2} \times S^{1}_{g_{k}})^{\prime}, \G_{k}^{\prime})$ is the Riemannian covering manifold of $( \RP^{3} \# \RP^{3}_{g_{k}}, \hat{\mathscr{G}}_{k} )$. Compared with $S^{2} \times S^{1}_{g_{k}}$ in Freedman's example, the $3$-manifold $(S^{2} \times S^{1}_{g_{k}})^{\prime} $ is obtained from $M_{g_{k}}$ by doing similar Dehn surgeries, but the number of Dehn surgeries is doubled. The smooth Riemannian metric $\G_{k}^{\prime}$ on $(S^{2} \times S^{1}_{g_{k}})^{\prime}$ is obtained by applying the cutoff function technique, see Section 6.1. If we let the radius $\varepsilon_{k}$ of solid tori in Dehn surgeries performed on $M_{g_{k}}$ be small enough, by applying Freedman's technique ( see Theorem \ref{Fre_estimation} ), we have
\begin{equation}
 \sys H_{2}( (S^{2} \times S^{1}_{g_{k}})^{\prime}, \G^{\prime}_{k}; \Z_{2} ) \geqslant C \tts g_{k},
\end{equation} 
where $C$ is a positive constant independent of $g_{k}.$ On the other hand, the surface $X_{k}$ either can be lifted to a nonseparating surface $\tilde{X}_{k}$ in $(S^{2} \times S^{1}_{g_{k}})^{\prime}$, or is doubly covered by a nonseparating surface $\tilde{X}_{k}$ in $( S^{2} \times S^{1}_{g_{k}} )^{\prime}.$ Hence we have
\begin{align*}
 \area_{\hat{\mathscr{G}}_{k}} (X_{k}) & \geqslant \frac{1}{2} \tts \area_{\G^{\prime}_{k}} (\tilde{X}_{k}) \\ 
        & \geqslant \frac{1}{2} \tts \sys H_{2}( (S^{2} \times S^{1}_{g_{k}})^{\prime}, \G^{\prime}_{k}; \Z_{2} ) \\
        & \geqslant \frac{1}{2}  C \tts g_{k}.
\end{align*}  
Therefore,
\begin{equation*}
 \sys H_{2}(\RP^{3} \# \RP^{3}_{g_{k}}, \hat{\mathscr{G}}_{k}; \Z_{2}) \geqslant s_{2} \tts g_{k},
\end{equation*}
where $s_{2}$ is a positive constant independent of $g_{k}$. 
\end{proof}

For volume, we have the following upper bound estimation.
\begin{proposition}
 When $k$ is large enough, the volume of 
 \begin{equation*}
  (\RP^{3} \# \RP^{3}_{g_{k}}, \hat{\mathscr{G}}_{k} ) 
 \end{equation*}
 satisfies
 \begin{equation}
  \vol_{\hat{\mathscr{G}}_{k}}(\RP^{3} \# \RP^{3}_{g_{k}}) \leqslant s_{3} \ts g_{k},
  \label{RP3_3}
 \end{equation}
 where $s_{3}$ is a positive constant independent of $g_{k}.$
\end{proposition}

\begin{proof}
To be the same as above, we use $( (S^{2} \times S^{1}_{g_{k}})^{\prime}, \G_{k}^{\prime} )$ to denote the two-sheeted Riemannian covering manifold of $( \RP^{3} \# \RP^{3}_{g_{k}}, \hat{\mathscr{G}}_{k} )$. When $k$ is large enough, the radius $\varepsilon_{k}$ of solid tori in Dehn surgeries on $M_{g_{k}}$ is small enough, then applying Freedman's technique ( see Theorem \ref{Fre_estimation} ), we have 
\begin{equation*}
 \vol_{\G^{\prime}_{k}}((S^{2} \times S^{1}_{g_{k}})^{\prime}) \leqslant C \tts g_{k},
\end{equation*}
where $C$ is a positive constant independent of $g_{k}.$ On the other hand, we have
\begin{equation*}
 \vol_{\hat{\mathscr{G}}_{k}}(\RP^{3} \# \RP^{3}_{g_{k}}) = \frac{1}{2} \tts \vol_{\G^{\prime}_{k}}( (S^{2} \times S^{1}_{g_{k}})^{\prime} ) .  
\end{equation*}
Hence the inequality (\ref{RP3_3}) holds.
\end{proof}

Based on above estimates (\ref{RP3_1}), (\ref{RP3_2}) and (\ref{RP3_3}), we have
\begin{align*}
 & \frac{\vol_{\hat{\mathscr{G}}_{k}}(\RP^{3} \# \RP^{3}_{g_{k}})}{\sys H_{1}(\RP^{3} \# \RP^{3}_{g_{k}}, \hat{\mathscr{G}}_{k}; \Z_{2}) \cdot \sys H_{2}( \RP^{3} \# \RP^{3}_{g_{k}}, \hat{\mathscr{G}}_{k}; \Z_{2})} \\
 & \leqslant \frac{s_{3} \tts g_{k}}{s_{1} \tts (\log{g_{k}})^{1/2} \cdot s_{2} \tts g_{k}}.
\end{align*}
Therefore, we have (\ref{semibundle}) by letting $k \to \infty.$

\section*{Acknowledgments}
I am grateful to my Ph.D. advisor Dr. Weiping Li for many helpful discusstions on the topic of systolic freedom. Also I appreciate Dr. Mikhail Katz for pointing out to me that E. Fetaya has another interpretation of Freedman's work .


\begin{thebibliography}{99}
\bibitem{Ag13} I. Agol, The virtual {H}aken conjecture. With an appendix by Agol, Daniel Groves, and Jason Manning. Doc. Math. 18 (2013), 1045--1087.
\bibitem{Ba02} I. Babenko, Forte souplesse intersystolique de vari\'et\'es ferm\'ees et de poly\`edres. Ann. Inst. Fourier (Grenoble) 52 (2002), no. 4, 1259--1284. 
\bibitem{Ba02b} I. Babenko, Loewner's conjecture, the Besicovitch barrel, and relative systolic geometry. Mat. Sb. 193 (2002), no. 4, 3--16; translation in Sb. Math. 193 (2002), no. 3--4, 473–-486. 
\bibitem{BaKa98} I. Babenko and M. Katz, Systolic freedom of orientable manifolds. Ann. Sci. \'Ecole Norm. Sup. (4) 31 (1998), no. 6, 787--809.
\bibitem{BaKaSu98} I. Babenko and M. Katz and A. Suciu, Volumes, middle-dimensional systoles, and {W}hitehead products. Math. Res. Lett. 5 (1998), no. 4, 461--471.
 \bibitem{BanKa03} V. Bangert and M. Katz, Stable systolic inequalities and cohomology products. Comm. Pure Appl. Math. 56 (2003), no. 7, 979--997. 
\bibitem{Ba86} C. Bavard, In\'egalit\'e isosystolique pour la bouteille de {K}lein. (French) [Isosystolic inequality for the Klein bottle] Math. Ann. 274 (1986), no. 3, 439--441. 
\bibitem{BeKa94} L. B{\'e}rard-Bergery and M. Katz, On intersystolic inequalities in dimension 3. Geom. Funct. Anal. 4 (1994), no. 6, 621--632. 
\bibitem{Bel13} M. Belolipetsky, On 2-systoles of hyperbolic 3-manifolds. Geom. Funct. Anal. 23 (2013), no. 3, 813--827. 
\bibitem{BeTho11} M. Belolipetsky and S. Thomson, Systoles of hyperbolic manifolds. Algebr. Geom. Topol. 11 (2011), no. 3, 1455--1469. 
\bibitem{Ber03} M. Berger, A panoramic view of Riemannian geometry. Springer-Verlag, Berlin, 2003. 
\bibitem{Bru08} M. Brunnbauer, On manifolds satisfying stable systolic inequalities. Math. Ann. 342 (2008), no. 4, 951--968. 
\bibitem{Bu82} P. Buser, A note on the isoperimetric constant. Ann. Sci. \'Ecole Norm. Sup. (4) 15 (1982), no. 2, 213--230. 
\bibitem{BuSa94} P. Buser and P. Sarnak, On the period matrix of a Riemann surface of large genus. With an appendix by J. H. Conway and N. J. A. Sloane. Invent. Math. 117 (1994), no. 1, 27--56. 
\bibitem{CheLi14} L. Chen and W. Li, Systoles of surfaces and 3-manifolds. Survey article (2014), to appear in Proceedings of the 2013 Midwest Geometry Conference, edited by: Weiping Li and Shihshu Walter Wei.
\bibitem{CroKa03} C. Croke and M. Katz, Universal volume bounds in Riemannian manifolds. Surveys in differential geometry, Vol. VIII (Boston, MA, 2002), 109--137, Surv. Differ. Geom., 8, Int. Press, Somerville, MA, 2003. 
\bibitem{FaMa12} B. Farb and D. Margalit, A primer on mapping class groups. Princeton Mathematical Series, 49. Princeton University Press, Princeton, NJ, 2012.
\bibitem{Fe11} E. Fetaya, Homological error correcting codes and systolic geometry. arXiv:1108.2886 (2011). 
\bibitem{Fr99} M. Freedman, {$\mathbb{Z}_{2}$}-systolic-freedom. Proceedings of the Kirbyfest (Berkeley, CA, 1998), 113--123 (electronic), Geom. Topol. Monogr., 2, Geom. Topol. Publ., Coventry, 1999. 
\bibitem{FrMeLu02} M. Freedman and D. Meyer and F. Luo, $\Z_{2}$-systolic freedom and quantum codes. Mathematics of quantum computation, Chapman \& Hall/CRC (2002): 287--320.
\bibitem{Gr83} M. Gromov, Filling Riemannian manifolds. J. Differential Geom. 18 (1983), no. 1, 1--147.
\bibitem{Gr92} M. Gromov, Systoles and intersystolic inequalities. Actes de la {T}able {R}onde de {G}\'eom\'etrie {D}iff\'erentielle ({L}uminy, 1992), 291--362, S\'emin. Congr., 1, Soc. Math. France, Paris, 1996. 
\bibitem{Gr07} M. Gromov, Metric structures for Riemannian and non-Riemannian spaces. Based on the 1981 French original. With appendices by M. Katz, P. Pansu and S. Semmes. Translated from the French by Sean Michael Bates. Reprint of the 2001 English edition. Modern Birkh\"auser Classics. Birkh\"auser Boston Inc., Boston, MA, 2007. 
\bibitem{Gu10} L. Guth, Metaphors in systolic geometry. Proceedings of the International Congress of Mathematicians. Volume II, 745--768, Hindustan Book Agency, New Delhi, 2010. 
\bibitem{Heb86} J. Hebda, The collars of a Riemannian manifold and stable isosystolic inequalities. Pacific J. Math. 121 (1986), no. 2, 339--356.
\bibitem{He04} J. Hempel, 3-Manifolds. Reprint of the 1976 original. AMS Chelsea Publishing, Providence, RI, 2004.
\bibitem{HeJa72} J. Hempel and W. Jaco, Fundamental groups of 3-manifolds which are extensions. Ann. of Math. (2) 95 (1972), 86--98.
\bibitem{Ka95} M. Katz, Counterexamples to isosystolic inequalities. Geom. Dedicata 57 (1995), no. 2, 195--206.
\bibitem{Ka02} M. Katz, Local calibration of mass and systolic geometry. Geom. Funct. Anal. 12 (2002), no. 3, 598--621.
\bibitem{Ka07} M. Katz, Systolic geometry and topology. With an appendix by Jake P. Solomon. Mathematical Surveys and Monographs, 137. American Mathematical Society, Providence, RI, 2007. 
\bibitem{Katzetal11} K. Katz, M. Katz, S. Sabourau, S. Shnider and S. Weinberger, Relative systoles of relative -essential 2-complexes, Algebraic $\&$ Geometric Topology 11 (2011), 101--121.
\bibitem{KaSa12} M. Katz and S. Sabourau, Hyperellipticity and systoles of Klein surfaces. Geom. Dedicata 159 (2012), 277--293. 
\bibitem{KaSa12b} M. Katz and S. Sabourau, Dyck's surfaces, systoles, and capacities. arXiv:1205.0188 [math.DG].
\bibitem{KaSu98} M. Katz and A. Suciu, Volume of Riemannian manifolds, geometric inequalities, and homotopy theory. Tel Aviv Topology Conference: Rothenberg Festschrift (1998), 113--136, 
Contemp. Math., 231, Amer. Math. Soc., Providence, RI, 1999.
\bibitem{KaSu01} M. Katz and A. Suciu, Systolic freedom of loop space. Geom. Funct. Anal. 11 (2001), no. 1, 60--73. 
\bibitem{LaLei14} G. Lakeland and C. Leininger, Systoles and Dehn surgery for hyperbolic 3-manifolds. Algebr. Geom. Topol. 14 (2014), no. 3, 1441--1460.
\bibitem{Na13} K. Nakamura, On isosystolic inequalities for $\mathbb{T}^{n}$, $\mathbb{R}P^{n}$, and $M^{3}$. Available at arXiv: 1306.1617 (2013).
\bibitem{PaiBa14} J. {\'A}lvarez Paiva and F. Balacheff, Contact geometry and isosystolic inequalities. Geom. Funct. Anal. 24 (2014), no. 2, 648--669. 
\bibitem{Pi97} C. Pittet, Systoles on {$S^1\times S^n$}. Differential Geom. Appl. 7 (1997), no. 2, 139--142. 
\bibitem{Pu52} P. Pu, Some inequalities in certain nonorientable {R}iemannian manifolds. Pacific J. Math. 2, (1952). 55--71.
\bibitem{Ro03} D. Rolfsen, Knots and Links,  Providence (R.I.) : American mathematical society, 2003.
\bibitem{Scha97} P. Schmutz Schaller, Extremal Riemann surfaces with a large number of systoles. Extremal Riemann surfaces (San Francisco, CA, 1995), 9--19, Contemp. Math., 201, Amer. Math. Soc., Providence, RI, 1997.
\bibitem{Thu98} W. Thurston, Hyperbolic structures on 3-manifolds, II: Surface groups and 3-manifolds which fiber over the circle. arXiv preprint math/9801045.
\bibitem{Zu97} L. Zulli, Semibundle decompositions of $3-$manifolds and the twisted cofundamental group. Topology Appl. 79 (1997), no. 2, 159--172. 
\end{thebibliography}
\end{document}